\RequirePackage{fix-cm}

\documentclass[smallextended]{svjour3}      

\smartqed  
\usepackage{graphicx}

\usepackage[T1]{fontenc}
\usepackage[latin1]{inputenc}

\usepackage{amssymb}
\usepackage{amsfonts}
\usepackage{amsmath}
\usepackage{wasysym}

\usepackage{epsfig}
\usepackage{epic}
\usepackage{eepic}
\usepackage{graphics}
\usepackage{graphicx}
\usepackage{color}
\usepackage{longtable}
\usepackage{subfigure}

\usepackage{pgfplots}

\usepackage{tikz}
\usetikzlibrary{calc,fadings,decorations.pathreplacing}
\usetikzlibrary{matrix,arrows}

\usetikzlibrary{positioning,fit,shapes,external,3d,patterns,spy}
\usetikzlibrary[shapes.arrows]
\usetikzlibrary{pgfplots.colormaps} 

\AtBeginDocument{%
  \paperwidth=\dimexpr
    1in + \oddsidemargin
    + \textwidth
    + 1in + \oddsidemargin
  \relax
  \paperheight=\dimexpr
    1in + \topmargin
    + \headheight + \headsep
    + \textheight
    + 1in + \topmargin
  \relax
  \usepackage[pass]{geometry}\relax
}

\newcommand{\ds}{\displaystyle}
\newcommand{\ts}{\textstyle}

\newcommand{\diag}{\operatorname{diag}}

\newcommand{\spann}{\operatorname{span}}

\newcommand{\Emb}{\operatorname{E}_\gamma}

\newcommand{\Nn}{{\mathbb N}}

\newcommand{\Rr}{{\mathbb R}}

\newcommand{\Zz}{{\mathbb Z}}

\newcommand{\Ttt}{\mathbf{T}}
\newcommand{\Ccc}{\mathbf{C}}
\newcommand{\Ddd}{\mathbf{D}}
\newcommand{\Mmm}{\mathbf{M}}

\newcommand{\PiL}{\Pi_{n,p}^L}
\newcommand{\Lagpol}{\mathcal{L}_{n,p} f}

\newcommand{\Lub}{\mathrm{Lisa}_{n,p}}
\newcommand{\Lubblack}{\mathrm{Lisa}_{n,p}^{\mathrm{b}}}
\newcommand{\Lubwhite}{\mathrm{Lisa}_{n,p}^{\mathrm{w}}}

\newcommand{\LubAbr}{\mathrm{Lisa}}

\newcommand{\Liscurve}{\gamma_{n,p}}

\newcommand{\Aa}{\mathcal{A}}
\newcommand{\Bb}{\mathcal{B}}

\newcommand{\Sub}{\mathcal{X}}

\newcommand{\equi}{\overset{\Lub}{\sim}}
\newcommand{\Tripi}{\Pi_{2n(n+p)}^{\mathrm{trig},L}}

\newcommand{\Lagbastrig}{ e_{ij} }
\newcommand{\LagbastrigE}{\Emb \! \left( \! \hat{T}_{i}(x) \hat{T}_{j}(y)\! \right)\! }

\newcommand{\Dx}[1]{\mathrm{d}#1}

\newcommand{\minus}{%
  \setbox0=\hbox{-}%
  \vcenter{%
    \hrule width\wd0 height \the\fontdimen8\textfont3%
  }%
}

\begin{document}

\title{Bivariate Lagrange interpolation at the node points of non-degenerate Lissajous curves}
\subtitle{}

\author{Wolfgang Erb        \and
        Christian Kaethner  \and
        Mandy Ahlborg       \and
        Thorsten M. Buzug
}

\institute{W. Erb, C. Kaethner, M. Ahlborg, T.M. Buzug \at
              Universit\"at zu L\"ubeck \\
	      Ratzeburger Allee 160\\
	      23562 L\"ubeck \\
              \email{erb@math.uni-luebeck.de} \\          
              \email{kaethner@imt.uni-luebeck.de} \\
              \email{ahlborg@imt.uni-luebeck.de} \\
              \email{buzug@imt.uni-luebeck.de}
}

\date{September 26, 2014}

\titlerunning{Lagrange interpolation at Lissajous nodes}
\authorrunning{W. Erb et al. }

\maketitle

\begin{abstract}
Motivated by an application in Magnetic Particle Imaging, we study bivariate Lagrange interpolation at the node points 
of Lissajous curves. The resulting theory is a generalization of the polynomial interpolation theory developed for a node set known as Padua points. 
With appropriately defined polynomial spaces, we will show that the node points of non-degenerate Lissajous
curves allow unique interpolation and can be used for quadrature rules in the bivariate setting. An explicit formula for the Lagrange polynomials
allows to compute the interpolating polynomial with a simple algorithmic scheme. Compared to the already established schemes of the Padua and Xu points, the numerical results for the proposed
scheme show similar approximation errors and a similar growth of the Lebesgue constant.


\end{abstract}

\section{Introduction}

A challenging task for multivariate polynomial interpolation is the construction of a suitable set of node points. Depending on the application, 
these node points should provide a series of favorable properties including a unique interpolation in given polynomial spaces, a slow growth of the Lebesgue constant and
simple algorithmic schemes that compute the interpolating polynomial. The construction of suitable point sets for multivariate interpolation 
has a long-standing history. For an overview, we refer to the survey articles \cite{GascaSauer2000b,GascaSauer2000} and the references therein.   
Examples of remarkable constructions in the bivariate setting are the point sets introduced by Morrow and Patterson \cite{MorrowPatterson1978}, Xu \cite{Xu1996}, as well as some 
generalizations of them~\cite{Harris2013}. A modification of the Morrow-Patterson points, introduced as Padua points~\cite{CaliariDeMarchiVianello2005}, is particularly
interesting for the purposes of this article. 

In some applications, the given data points are lying on subtrajectories of the euclidean space. In this case, aside from the above mentioned favorable properties, it is 
mandatory that the node points are part of these trajectories. Lissajous curves are particularly interesting examples for us, as they are used as a sampling path in a young
medical imaging technology called Magnetic Particle Imaging (MPI)~\cite{Gleich2005Nature}.

In MPI, the distribution of superparamagnetic iron oxide nanoparticles is reconstructed by measuring the magnetic response of the particles. 
The measurement process is based on the combination of various magnetic fields that generate and move a magnetic field free point through a region of interest. 
Although different trajectories are possible, this movement is typically performed in form of 
a Lissajous curve~\cite{Knopp2009PhysMedBio}. The reconstruction of the particle density from the data on the Lissajous trajectory is currently done in a very
straight forward way, either by solving a system of linear equations based on a pre-measured system matrix or directly from the measurement
data~\cite{Gruettner2013BMT}. By using multivariate polynomial interpolation on the nodes of the sampling path, i.e. the Lissajous curve, we hope to obtain a further improvement
in the reconstruction process. 

Of the node points mentioned above, the Padua points, as described in \cite{BosDeMarchiVianelloXu2006}, are the ones with the strongest relation to Lissajous curves. 
They can be characterized as the node points of a particular degenerate Lissajous figure. Moreover, they satisfy a series of remarkable 
properties: they can be described as an affine variety of a polynomial ideal~\cite{BosDeMarchiVianelloXu2007}, they form a particular Chebyshev lattice~\cite{CoolsPoppe2011} and they 
allow a unique interpolation in the space $\Pi_n$ of bivariate polynomials of degree $n$~\cite{BosDeMarchiVianelloXu2006}. Furthermore, a simple formula for the Lagrange polynomials 
is available and the Lebesgue constants are growing slowly as $\mathcal O\left(\log^2 n \right)$~\cite{BosDeMarchiVianelloXu2006}. 

The aim of this article is to develop, similar to the Padua points, an interpolation theory for node points on Lissajous curves. To this end,
we extend the generating curve approach as presented in \cite{BosDeMarchiVianelloXu2006} to particular families of \mbox{Lissajous} curves in $[-1,1]^2$. 
In this article, we will focus on the node points of non-degenerate Lissajous curves, which are important for the application in MPI~\cite{Kaethner2014IEEE}.
Not all of the above mentioned properties of the Padua points will be carried over to the node points of Lissajous figures. However, the resulting theory will have some interesting resemblences, not only
to the theory of the Padua points, but also to the Xu points. 
 
We start our investigation by characterizing the node points $\Lub$ of non-degenerate Lissajous curves. Based on the 
node points $\Lub$, we will derive suitable quadrature formulas for integration with product Chebyshev weight functions. 
Next, we will provide the main theoretical results on bivariate interpolation based on the points $\Lub$. 
We will show that the points $\Lub$ allow unique interpolation in a properly defined space $\PiL$ of bivariate polynomials. 
Further, we will derive a formula for the fundamental polynomials of Lagrange interpolation.
This explicit formula allows to compute the interpolating polynomial with a simple algorithmic scheme similar to the one of 
the Padua points \cite{CaliariDeMarchiVianello2008}. We conclude 
this article with some numerical tests for the new bivariate interpolating schemes. Compared to the established interpolating schemes of the Padua and Xu points, the novel interpolation schemes show 
similar approximation errors and a similar growth of the Lebesgue constant.   

\section{The node points of non-degenerate Lissajous curves} \label{sec:Lissajous}

In this article, we consider $2\pi$-periodic Lissajous curves of the form
\begin{equation} \label{eq:generatingcurve}
 \Liscurve: \Rr \to \Rr^2, \quad \Liscurve(t) = \Big( \sin(nt), \, \sin((n+p)t) \Big),
\end{equation}
where $n$ and $p$ are positive integers such that $n$ and $n+p$ are relatively prime. Based on the calculations in \cite{BogleHearstJonesStoilov1994} (see also \cite{Lamm1997}), the 
Lissajous curve $\Liscurve$ is non-degenerate if and only if $p$ is odd. In this case, $\gamma_{n,p}: [0,2\pi) \to \Rr^2$ is an immersed plane curve with precisely 
$2 n (n+p) -2 n - p$ self-intersection points. In the following, we will always assume that $p$ is odd and sample the Lissajous curve $\gamma_{n,p}$ along the $4n(n+p)$ equidistant points 
\[ t_k := \frac{2\pi k}{ 4 n (n+p)}, \quad k = 1, \ldots, 4n(n+p).\]
In this way, we get the following set of Lissajous node points:
\begin{equation}
\Lub := \Big\{ \Liscurve (t_k): \quad k = 1, \ldots, 4n(n+p) \Big\}. \label{def:Luebeckpoints}
\end{equation}

To characterize the set $\Lub$, we divide $\Liscurve (t_k)$ for the even and odd integers $k$. For this decomposition, we use the fact that $n$ and $n+p$ are relatively prime. 
Then, if $n$ is odd, every odd integer $k$ can be written as $k = (2i+1) n + 2j(n+p)$ with $i,j \in \Zz$. If $k$ is even, we can write  
$k = 2 i n  + (2j+1)(n+p) $ with $i,j \in \Zz$. If $n$ is even, the same holds with the roles of $n$ and $n+p$ switched. In this way, we get the decomposition 
$\Lub = \Lubblack \cup \Lubwhite$ with the sets
\begin{align} 
\Lubblack &:= \left\{ \Liscurve \left( \frac{(2i+1)n + 2j (n+p) }{ 4 n (n+p)} 2\pi \right): \quad i,j \in \Zz \right\}, \label{def:Luebeckpointsb}\\
\Lubwhite &:= \left\{ \Liscurve \left( \frac{ 2i n + (2j+1)(n+p)}{ 4 n (n+p)} 2\pi \right): \quad i,j \in \Zz \right\}. \label{def:Luebeckpointsr}
\end{align}

\begin{figure}[htb]
	\centering
	\subfigure[	Lissajous figure $\gamma_{2,1}$, $|\LubAbr_{2,1}| = 17$.]{\includegraphics[scale=0.75]{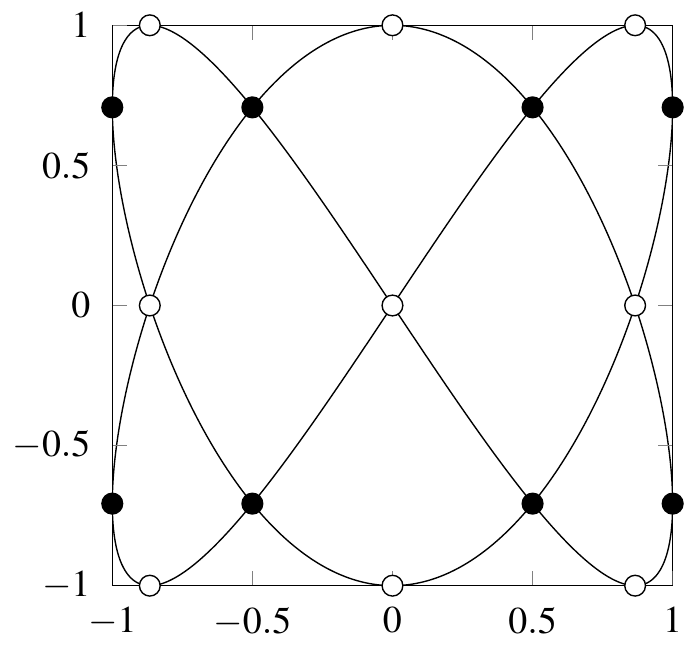}}
	\hfill	
	\subfigure[	Lissajous figure $\gamma_{2,3}$, $|\LubAbr_{2,3}| = 27$.]{\includegraphics[scale=0.75]{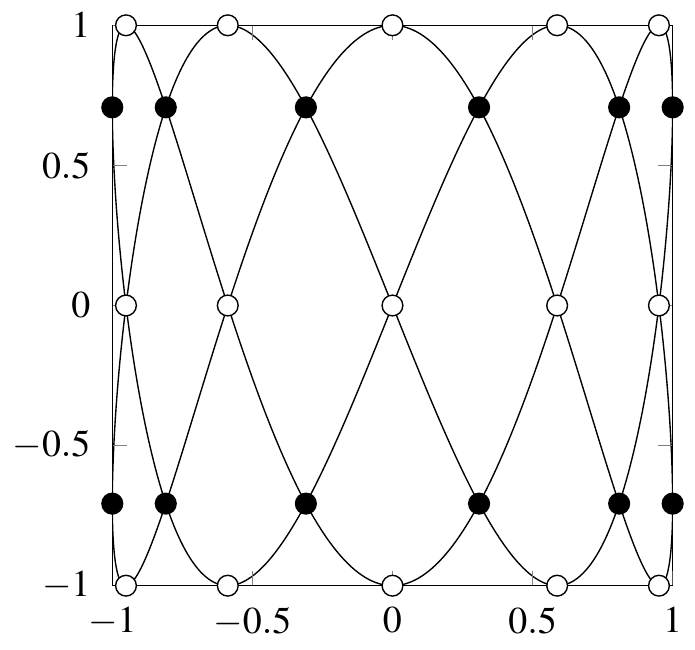}}
  	\caption{Illustration of non-degenerate Lissajous curves $\Liscurve$. 
  	The node points $\Lub$ of $\Liscurve$ are arranged on two different grids (black, white) corresponing to the sets $\Lubblack$ and $\Lubwhite$. 
  	}
	\label{fig:lissajous}
\end{figure}

Two examples of Lissajous curves $\Liscurve$ with the corresponding node points $\Lubblack$ and $\Lubwhite$ are illustrated in Figure \ref{fig:lissajous}.
To get a compact representation of $\Lubblack$ and $\Lubwhite$, we use the following notation for the Gau{\ss}-Lobatto points:
\begin{equation}
	z_k^{n} := \cos\left(\frac{k\pi}{n}\right), \quad n \in \Nn, \; k \in \Zz.
	\label{eq:chebLob}
\end{equation}
Then, evaluating the points \eqref{def:Luebeckpointsb} and \eqref{def:Luebeckpointsr} explicitly for the Lissajous curve \eqref{eq:generatingcurve}, 
we get the following characterization:
\begin{align}
\Lubblack & = \left\{ (-1)^{i+j}  \Big( z_{(2i+1)p}^{2(n+p)}, z_{2 j p}^{2n}\Big): \quad \begin{array}{l} i = 0, \ldots, n+p-1 \\ j = 0, \ldots, n \end{array} 
\right\}, \label{eq:Luebeckpointsb} \\
\Lubwhite & = \left\{ (-1)^{i+j} \Big( z_{2i p}^{2(n+p)}, z_{(2 j + 1) p}^{2n}\Big): \quad \begin{array}{l} i = 0, \ldots, n+p \\ j = 0, \ldots, n-1 \end{array}
\right\}. \label{eq:Luebeckpointsr}
\end{align}

Since $p$ is assumed to be odd and relatively prime to $n$, $p$ is relatively prime to $2n$ as well as to $2(n+p)$. Therefore, by rearranging the points, we can drop the number
$p$ in the lower indices of the Gau{\ss}-Lobatto points in \eqref{eq:Luebeckpointsb} and 
\eqref{eq:Luebeckpointsr}. Due to the point symmetry of the Lissajous curve $\Liscurve$, the term $(-1)^{i+j}$ which preceeds the points in \eqref{eq:Luebeckpointsb} and 
\eqref{eq:Luebeckpointsr} can also be dropped by further rearrangement. This leads to the following simple characterization of the point sets $\Lubblack$ and $\Lubwhite$:
\begin{align}
\Lubblack & = \left\{ \Big( z_{2i'+1}^{2(n+p)}, z_{2 j'}^{2n}\Big): \quad \begin{array}{l} i' = 0, \ldots, n+p-1 \\ j' = 0, \ldots, n \end{array} 
\right\},  \label{eq:Luebeckpointsb2} \\
\Lubwhite & = \left\{ \Big( z_{2i' }^{2(n+p)}, z_{2 j' + 1}^{2n}\Big): \quad \begin{array}{l} i' = 0, \ldots, n+p \\ j' = 0, \ldots, n-1 \end{array}
\right\}. \label{eq:Luebeckpointsr2}
\end{align}

With this characterization, we can also divide the points $\Lub$ into the sets $\Lub^{\mathrm{int}}$ and $\Lub^{\mathrm{out}}$ 
denoting the points lying in the interior and on the boundary of the square $[-1,1]^2$
respectively. We have
\begin{align*}
\Lub^{\mathrm{int}} &:=\left\{ \Big( z_{2i'+1}^{2(n+p)}, z_{2 j'}^{2n}\Big): \quad \begin{array}{l} i' = 0, \ldots, n+p-1 \\ j' = 1, \ldots, n-1 \end{array} \right\} \\ 
                    & \hspace{3cm} \cup \left\{ \Big( z_{2i' }^{2(n+p)}, z_{2 j' + 1}^{2n}\Big): \quad \begin{array}{l} i' = 1, \ldots, n+p-1 \\ j' = 0, \ldots, n-1 \end{array}
\right\}, \\
\Lub^{\mathrm{out}} &:=\left\{ \Big( z_{2i'+1}^{2(n+p)}, \pm 1 \Big): \quad i' = 0, \ldots, n+p-1 
\right\} \\ & \hspace{3cm} \cup \left\{ \Big( \pm 1, z_{2j'+1}^{2n}\Big): \; j' = 0, \ldots, n-1 \right\}.
\end{align*}

From the representation of the Lisa points in \eqref{eq:Luebeckpointsb2} and \eqref{eq:Luebeckpointsr2}, it is possible to count the number of points in the different sets.
They are listed in Table \ref{tab:1}. 

\begin{table}[htb] 
 \caption{Cardinality of the different Lisa sets.} \label{tab:1} 
 \begin{center}
 \begin{tabular}[t]{ll} \hline \noalign{\smallskip}
  Set & Number of elements \\ \noalign{\smallskip}\hline \noalign{\smallskip}
  $\Lub$ & $2n(n+p) + 2n + p$   \\ \noalign{\smallskip}
  $\Lubblack$ & $(n+1) (n+p)$\\ \noalign{\smallskip}
  $\Lubwhite$ & $n (n+p+1)$\\ \noalign{\smallskip}
  $\Lub^{\mathrm{int}}$ & $2n(n+p) - 2 n - p$ \\ \noalign{\smallskip}
  $\Lub^{\mathrm{out}}$ & $4n+2p$  \\ \hline 
  \end{tabular}    
  \end{center}
\end{table}

From the representation in \eqref{def:Luebeckpointsb} and \eqref{def:Luebeckpointsr} and its identification in \eqref{eq:Luebeckpointsb} and 
\eqref{eq:Luebeckpointsr}, we can deduce that 
\begin{align*} 
\Liscurve \left( \frac{(2i+1)n + 2j (n+p) }{ 4 n (n+p)} 2\pi \right) &= \Liscurve \left( \frac{(2i+1)n - 2j (n+p) }{ 4 n (n+p)} 2\pi \right),\\
\Liscurve \left( \frac{ 2i n + (2j+1)(n+p)}{ 4 n (n+p)} 2\pi \right) &= \Liscurve \left( \frac{ - 2i n + (2j+1)(n+p)}{ 4 n (n+p)} 2\pi \right)
\end{align*}
holds for all $i,j \in \Zz$. Moreover, in \eqref{def:Luebeckpointsb} and \eqref{def:Luebeckpointsr} the boundary points are represented by $j \in n \Zz$ and $i \in (n+p) \Zz$, respectively. 
Thus, for interior points in $\Lubblack \cap \Lub^{\mathrm{int}}$, i.e. all points in \eqref{def:Luebeckpointsb} satisfying $j \neq n \Zz$, there exist
at least two different $1 \leq k,k' \leq 4n(n+p)$ in \eqref{def:Luebeckpoints} that represent the same point. The same holds for all interior points in the second set $\Lubwhite$. 

Therefore, all points in $\Lub^{\mathrm{int}}$ are self-intersection points of the Lissajous curve $\Liscurve$. Since $|\Lub^{\mathrm{int}}| = 2n(n+p) - 2 n - p$ corresponds
to the total number of self-intersection points of a non-degenerate Lissajous curve (see \cite{BogleHearstJonesStoilov1994}), we can conclude that $\Lub^{\mathrm{int}}$ is
precisely the set of all self-intersection points of the Lissajous curve $\Liscurve$. Finally, since $2 | \Lub^{\mathrm{int}}| + |\Lub^{\mathrm{out}}| = 4n(n+p)$, we can also
conclude that there are exactly two different $1 \leq k,k' \leq 4n(n+p)$ that represent 
the same point in $\Lub^{\mathrm{int}}$ and that every point in $\Lub^{\mathrm{out}}$ is described by exactly one $1 \leq k\leq 4n(n+p)$ in \eqref{def:Luebeckpoints}.

In order to identify the different integers $k$ in \eqref{def:Luebeckpoints} that describe the same
point $\Aa \in \Lub$, we introduce for $k,k' \in \Zz$ the equivalence relation 
\[ k \equi k' \quad \Leftrightarrow \quad \Liscurve (t_k) = \Liscurve (t_{k'}).\]
We say that $k \in \Zz$ belongs to the equivalence class $[\Aa]$, $\Aa \in \Lub$, if $\Liscurve(t_k) = \Aa$. 
Therefore, by the above argumentation, there is exactly one ${1 \leq k \leq 4n(n+p)}$ in the equivalence class $[\Aa]$ if $\Aa \in \Lub^{\mathrm{out}}$
and exactly two if $\Aa\in \Lub^{\mathrm{int}}$. 

\begin{remark}
There are some remarkable relations between the Lisa, Padua and Xu points. In formal terms, if $p = 0$ in the characterization \eqref{eq:Luebeckpointsb2} and \eqref{eq:Luebeckpointsr2}
of the Lisa points, the points $\LubAbr_{n,0}$ correspond with the even Xu points $\mathrm{XU}_{2n}$ as defined in \cite{Xu1996}. Moreover, if $p = \frac{1}{2}$ 
in \eqref{eq:Luebeckpointsb2} and \eqref{eq:Luebeckpointsr2}, we obtain the even Padua points $\mathrm{PD}_{2n}$ of the second family 
(see \cite{CaliariDeMarchiVianello2008} and \eqref{eq:Paduab}, \eqref{eq:Paduaw} in Section \ref{sec:numericalpart}) with a slight adjustment in the range of the indices. 
A further comparison of these three point sets in terms of numerical simulations is given in the last section
of this article. Finally we would like to add that the Lisa points, similarly to the Padua points, can be considered as two-dimensional Chebyshev lattices of rank 1 (see \cite{CoolsPoppe2011}). 
\end{remark}

\section{Quadrature formulas based on the Lissajous node points}

In this section, we study quadrature rules for bivariate integration defined by point evaluations at the points $\Lub$. As underlying polynomial spaces in $\Rr^2$, we consider
\[ \Pi_n = \spann\{T_{i}(x) T_j(y):\; i+j \leq n\},\]
where $T_i(x)$ denotes the Chebyshev polynomial $T_i(x) = \cos (i \arccos x)$ of the first kind. It is well-known (cf. \cite{Xu1996}) that $\{ T_i(x) T_j(y):\; i+j \leq n\}$ is an orthogonal basis
of $\Pi_n$ with respect to the inner product 
\begin{equation} \label{eq:scalarproduct}
 \langle f,g \rangle := \frac{1}{\pi^2} \int_{-1}^1 \int_{-1}^1 f(x,y) \overline{g(x,y)} \frac{1}{\sqrt{1-x^2}} \frac{1}{\sqrt{1-y^2}} \Dx{x} \Dx{y}.
\end{equation}
The corresponding orthonormal basis is given by $\{\hat{T}_i(x) \hat{T}_j(y): i+j \leq n \}$, where
\begin{equation*} \label{eq:normalizedpolynomials} \hat{T}_i(x) = \left\{ \begin{array}{ll} 
                          1, & \quad \text{if $i = 0$},  \\
                          \sqrt{2} T_i(x), & \quad \text{if $i \neq 0$}.
                          \end{array} \right.
\end{equation*}
Using the trajectory $\Liscurve$, it is possible to reduce a double integral of the form used in $\eqref{eq:scalarproduct}$ into a single integral for a large class of bivariate polynomials.
\begin{lemma} \label{lem:int2Dto1D}
For all polynomials $ P \in \Pi_{8n+4p-1}$ with $ \langle P, T_{2(n+p)}(x)T_{2n}(y) \rangle = 0$, the following formula holds:
\begin{equation} \label{eq:int2Dto1D}
\frac{1}{\pi^2} \int_{-1}^1 \int_{-1}^1 P(x,y) \frac{1}{\sqrt{1-x^2}} \frac{1}{\sqrt{1-y^2}} \Dx{x} \Dx{y} = \frac1{2\pi} \int_0^{2\pi} P(\Liscurve(t)) \Dx{t}. \end{equation}
\end{lemma}

\begin{proof}
We check \eqref{eq:int2Dto1D} for all basis polynomials $T_i(x) T_j(y)$ in the space $\Pi_{8n+4p-1}$. For the left hand side of \eqref{eq:int2Dto1D} we 
get the value $1$ if $(i,j) = (0,0)$ and $0$ otherwise. For the right hand side of \eqref{eq:int2Dto1D} we get also $1$ if $(i,j) = (0,0)$. For $(i,j) \neq (0,0)$ we get for $P(x,y) = T_i(x) T_j(y)$
the expression
\begin{align*}
\frac1{2\pi} \int_0^{2\pi} P(\Liscurve(t)) \Dx{t} &= \frac1{2\pi} \int_0^{2\pi} T_i(\sin(nt)) T_j( \sin((n+p)t)) \Dx{t} \\
&= \frac1{2\pi} \int_0^{2\pi} \cos \left( i n t - i \frac{\pi}2 \right) \cos \left( j (n+p) t - j \frac{\pi}2 \right) \Dx{t}.
\end{align*}
We now determine for which indices $(i,j)$ this integral is different from $0$. This is only the case if $in = j(n+p)$ and $i-j$ is even. Since the numbers $n$ and $n+p$ are relatively prime, this can
only be the case if $i = k(n+p)$, $j = nk$ and $k \in 2 \Nn$ is an even number. We see that the smallest possible value for $k$ is $k = 2$ and the second smallest is $k = 4$. 
Furthermore, the sum of the respective indices is given by $i + j = (2n+p)k$. 
Therefore, we can conclude that for all indices $(i,j)$ satisfying $i+j = 1, \ldots, 4n+2p-1$ and $i+j = 4n+2p+1, \ldots, 8n + 4p -1$ the right hand side of \eqref{eq:int2Dto1D} vanishes.
If $i+j = 4n+2p$, the above integral is nonzero only if $i = 2(n+p)$ and $j = 2n$. \qed
\end{proof}

To get a quadrature formula supported on the points $\Lub$, we define a suitable polynomial subspace
\begin{align*}
\Pi_{n,p}^Q = \spann\{T_{i}(x) T_j(y):\; (i,j) \in \Gamma_{n,p}^Q\}
\end{align*}
with the index set $\Gamma_{n,p}^Q \subset \Nn_0^2$ given by
\begin{align*}
 \Gamma_{n,p}^Q := \Big\{(i,j): \; i\!+\!j \leq 4n \!-\! 1 \Big\} & \cup \! \bigcup_{m = 0}^{4p-1} \! \left \{(i,j): \; i\!+\!j = 4n \! +\! m,\; j < \frac{n(4p \!-\! m)}{p} \right\}.
\end{align*}
Note that the particular index $(i,j) = (2(n+p),2n)$ is not included in $\Gamma_{n,p}^Q$ and that Lemma \ref{lem:int2Dto1D} is applicable for all polynomials 
$P \in \Pi_{n,p}^Q$. An example of the index set $\Gamma_{n,p}^Q$ is shown in Figure \ref{fig:bullets_cross}. Clearly, the polynomial space $\Pi_{n,p}^Q$ 
satisfies $\Pi_{4n-1} \subset \Pi_{n,p}^Q \subset \Pi_{4n+4p-1}$ and the dimension of 
$\Pi_{n,p}^Q$ can be computed as 
\[ \dim \Pi_{n,p}^Q = |\Gamma_{n,p}^Q| = 8n(n+p) + 4 n - 2(p-1) = 4(|\Lub|-n-p) - 2(p-1).\]

\begin{figure}[htb]
	\centering
	\includegraphics[scale=0.75]{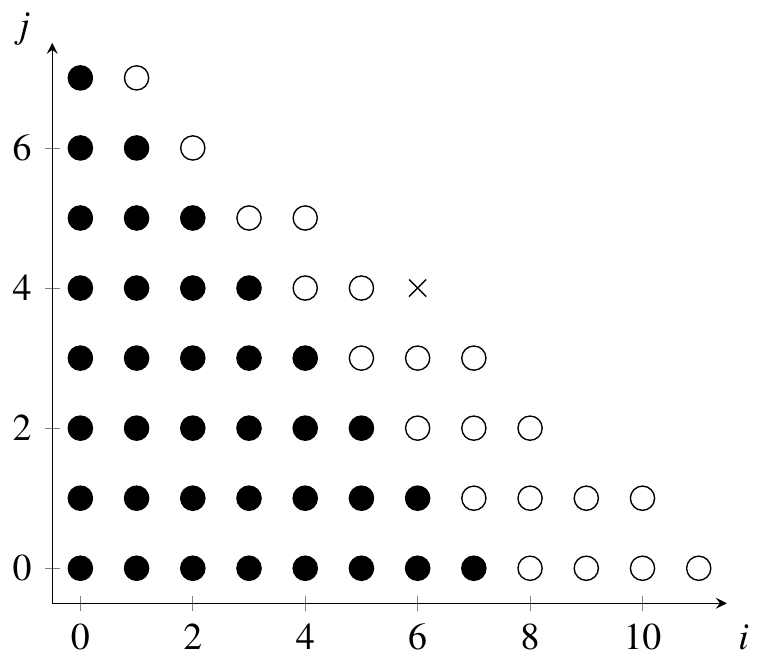}
	\caption{Illustration of the index set $\Gamma_{2,1}^Q$ with black and white bullets. We have $|\Gamma_{2,1}^Q| = 56$ black and white bullets. The black bullets correspond to indices describing the polynomial space $\Pi_{4n-1}$. The black cross is not contained in $\Gamma_{2,1}^Q$. It corresponds to the special index $(i,j) = (6,4)$ appearing in Lemma \ref{lem:int2Dto1D}.} 
	\label{fig:bullets_cross}
\end{figure}

For points $\Aa \in \Lub$, we define the quadrature weights
\[w_{\Aa} := \left\{ \begin{array}{ll} \ds \frac{1}{4n(n+p)}, \quad & \text{if $\Aa \in \Lub^\mathrm{out}$}, \\
                             \ds \frac{2}{4n(n+p)}, \quad & \text{if $\Aa \in \Lub^\mathrm{int}$}.
                 
                \end{array}
\right.\]
Then, we get the following quadrature rule based on the node set $\Lub$: 

\begin{theorem} \label{thm:quadratureruleluebeck}
For all $P \in \Pi_{n,p}^Q$ the quadrature formula
\begin{equation} \label{eq:quadratureruleluebeck}
\frac{1}{\pi^2} \int_{-1}^1 \int_{-1}^1 P(x,y) \frac{1}{\sqrt{1-x^2}} \frac{1}{\sqrt{1-y^2}} \Dx{x} \Dx{y} = \sum_{\Aa \in \Lub} w_\Aa P(\Aa) 
\end{equation}
is exact.
\end{theorem}

\begin{proof}
For all trigonometric $2\pi$-periodic polynomials $q$ of degree less than $4n(n+p)$, the following composite trapezoidal quadrature rule is exact:
\[ \frac{1}{2\pi} \int_0^{2\pi} q(t) dt = \frac1{4n(n+p)} \sum_{k=1}^{4n(n+p)} q\left( t_k \right).\]
Since $\Pi_{n,p}^Q \subset \Pi_{8n+4p-1}$ and $ \Pi_{n,p}^Q \perp T_{2(n+p)}(x)T_{2n}(y)$, we have by Lemma \ref{lem:int2Dto1D} the identity
\[
\frac{1}{\pi^2} \int_{-1}^1 \int_{-1}^1 P(x,y) \frac{1}{\sqrt{1-x^2}} \frac{1}{\sqrt{1-y^2}} \Dx{x} \Dx{y} = \frac1{2\pi} \int_0^{2\pi} P(\Liscurve(t)) \Dx{t}.
\]
Thus, if we show that for $P \in \Pi_{n,p}^Q$ the trigonometric polynomial $P(\Liscurve(t))$ is of degree less than $4n(n+p)$, we immediately get the quadrature formula
\begin{align*}
\frac{1}{\pi^2} \int_{-1}^1 \int_{-1}^1 P(x,y) \frac{1}{\sqrt{1-x^2}} \frac{1}{\sqrt{1-y^2}} \Dx{x} \Dx{y} &= \frac1{4n(n+p)} \sum_{k=1}^{4n(n+p)} P ( \Liscurve (t_k)) 
 \\ &= \sum_{\Aa \in \Lub} w_\Aa P(\Aa).
\end{align*}
To finish the proof we consider the representation of the polynomial $P \in \Pi_{n,p}^Q$ in the orthogonal basis $\{T_{i}(x) T_j(y):\; (i,j) \in \Gamma_{n,p}^Q \}$ and get
\begin{align*}
P(\Liscurve(t)) &= \sum_{(i,j) \in \Gamma_{n,p}^Q} a_{ij} T_i(\sin n t) T_j(\sin (n+p)t) \\
               &= \sum_{(i,j) \in \Gamma_{n,p}^Q} a_{ij} \cos \left(int - i \ts \frac{\pi}{2}\right) \cos\left( j (n+p) t - j \ts \frac{\pi}{2}\right)
\end{align*}
for some coefficients $a_{ij} \in \Rr$.
In order for the trigonometric polynomials in this formula to have a degree less than $4n(n+p)$, the indices $(i,j)$ have to satisfy the condition
\[ (i+j) n + jp < 4n(n+p).\]
In the case that $i+j < 4n$, we have $(i+j)n + jp \leq (i+j)n + 4np < 4n(n+p)$ and the condition above is satisfied.\\
In the case that $i + j = 4n+m$ with $0 \leq m \leq 4p-1$, we have
$(i+j)n + jp = 4n^2 + m n + j p < 4n(n+p)$ and the condition above is satisfied for all $j$ with $j < \frac{n(4p-m)}{p}$. 
By definition, this condition is exactly satisfied for all indices $(i,j) \in \Gamma_{n,p}^Q$ and therefore for all polynomials $P \in \Pi_{n,p}^Q$. \qed
\end{proof}

\begin{remark}
Lemma \ref{lem:int2Dto1D} and Theorem \ref{thm:quadratureruleluebeck} are generalizations of corresponding results proven in \cite{BosDeMarchiVianelloXu2006} for the Padua points. 
An analogous formula also exists for the Xu points (see \cite{MorrowPatterson1978,Xu1996}). Furthermore, the cardinality of the Xu points $\mathrm{XU}_{2n}$ is 
known to be minimal for exact integration of bivariate polynomials in $\Pi_{4n-1}$ with respect to a product Chebyshev weight function (see \cite{Moeller1976,Xu1996}). 
Since $|\Lub| > |\mathrm{XU}_{2n}| = 2 n (n+1)$, this is not the case for the Lisa points. On the other hand, as illustrated in Figure \ref{fig:bullets_cross}, the space $\Pi_{n,p}^Q$, for which 
\eqref{eq:quadratureruleluebeck} is exact, shows a remarkable asymmetry. As for multivariate interpolation, the construction of suitable nodes for cubature rules has a 
long history. For an overview, we refer to the survey article \cite{Cools1997}.
\end{remark}

\section{Interpolation on the Lissajous node points}

Given the quadrature formulas of the last section, we now investigate bivariate interpolation at the points $\Lub$. 
The corresponding interpolation problem can be formulated as follows: for given 
function values $f_\Aa \in \Rr$, $\Aa \in \Lub$, we want to
find a unique bivariate interpolating polynomial $\Lagpol$ such that 
\begin{equation}\label{eq:interpolationproblem}
\Lagpol (\Aa) = f_\Aa \quad \text{for all}\; \Aa \in \Lub.
\end{equation}
To set this problem correctly, we have to fix an underlying interpolation space. This space is linked to $\Pi_{n,p}^Q$ and defined as 
\begin{align*}
\Pi_{n,p}^L := \spann\{T_{i}(x) T_j(y):\; (i,j) \in \Gamma_{n,p}^L\}
\end{align*}
on the index set
\[
 \Gamma_{n,p}^L := \Big\{(i,j): \; i+j \leq 2n \Big\} \cup \bigcup_{m = 1}^{2p-1} \left \{(i,j): \; i+j = 2n+m,\; j < \frac{n(2p-m)}{p} \right\}.
\]

\begin{figure}[htb]
	\centering
	\subfigure[	Index set $\Gamma_{2,1}^L$ with $|\Gamma_{2,1}^L| = 17$.]{\includegraphics[scale=0.75]{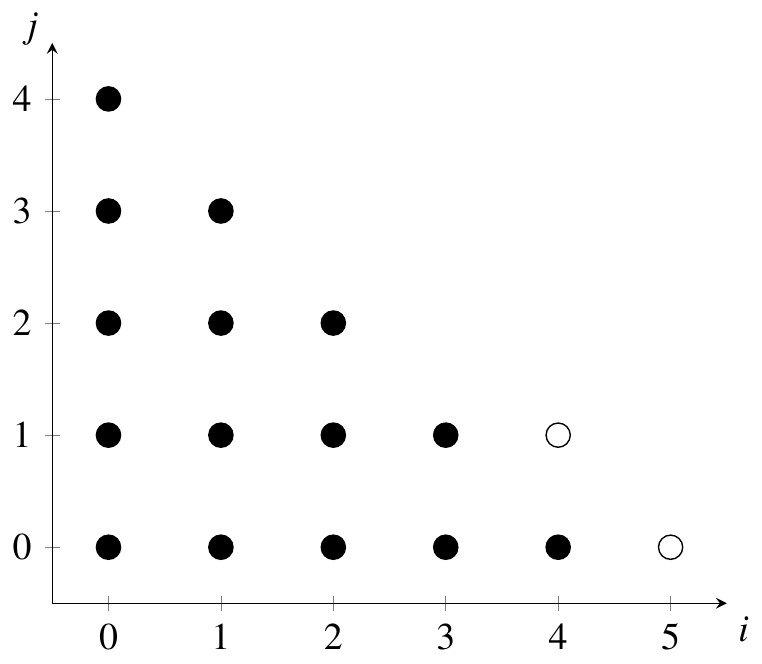}}
	\hfill	
	\subfigure[	Index set $\Gamma_{2,3}^L$ with $|\Gamma_{2,3}^L| = 27$.]{\includegraphics[scale=0.75]{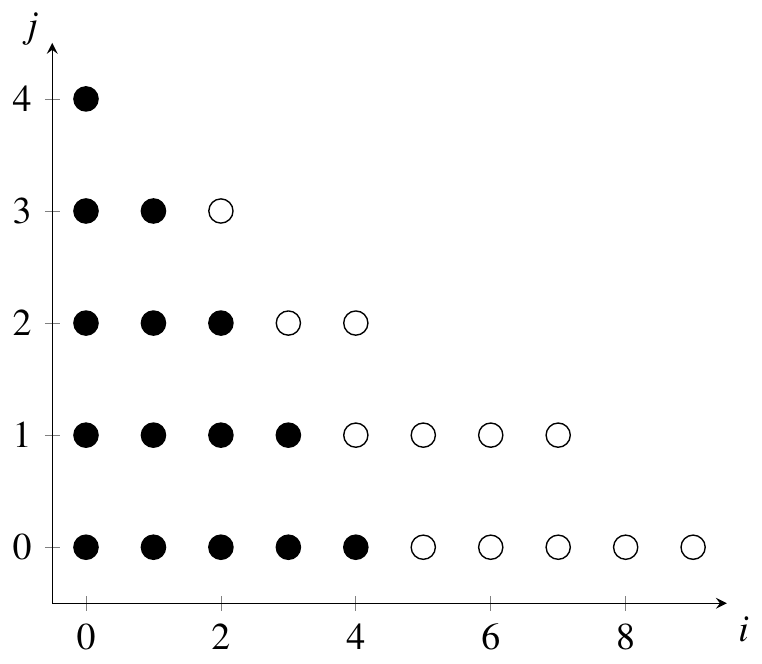}}
  	\caption{Illustration of the index sets $\Gamma_{n,p}^L$ for $n = 2$, $p = 1$ and $n=2$, $p=3$. The black bullets correspond to indices describing the polynomial space $\Pi_{2n}$.}
	\label{fig:twobullets}
\end{figure}

Examples of sets $\Gamma_{n,p}^L$ with different values of $p$ are given in Figure \ref{fig:twobullets}. 
The reproducing kernel $K_{n,p}^L: \Rr^2 \times \Rr^2 \to \Rr$ of the polynomial space $\Pi_{n,p}^L$ is given as
\[K_{n,p}^L(x_\Aa,y_\Aa; x_\Bb, y_\Bb) = \sum_{(i,j) \in \Gamma_{n,p}^L} \hat{T}_i(x_\Aa) \hat{T}_i(x_\Bb) \hat{T}_j(y_\Aa) \hat{T}_j(y_\Bb). \]
It is straightforward to check that the kernel $K_{n,p}^L$ has the reproducing property
\[ \langle P, K_{n,p}^L(x,y; \cdot) \rangle = P(x,y)\]
for all polynomials $P \in \Pi_{n,p}^L$. We have $\Pi_{2n} \subset \Pi_{n,p}^L \subset \Pi_{2(n+p)-1}$. The dimension of the polynomial space $\Pi_{n,p}^L$ is given as
\begin{align*}
\dim (\Pi_{n,p}^L) = |\Gamma_{n,p}^L| &= \frac{(2n+1)(2n+2)}{2} + \sum_{m=1}^{2p-1} \left\lceil \frac{n(2p-m)}{p} \right\rceil \\
                                      &=  2n^2 + n (2p+2) + p = |\Lub|.  
\end{align*}
Therefore, the dimension $\dim(\Pi_{n,p}^L)$ of the polynomial space $\Pi_{n,p}^L$ corresponds precisely to the number of distinct points in $\Lub$.

Soon, we will deduce a formula for the fundamental polynomials of Lagrange interpolation with respect to the points in $\Lub$ and show that the
interpolation problem \eqref{eq:interpolationproblem} has a unique solution.
To this end, we investigate an isomorphism between the polynomial space $\PiL$ and the subspace

\begin{equation} \label{def:spacetrigonometricpolynomials}
 \Tripi \!:= \left\{  q \in\! \Pi_{2n(n+p)}^{\mathrm{trig}}\!\!:  
\; \text{$q(t_k) = q(t_{k'})$ for all $k,k'$ with $k \!\!\equi\!\! k'$} \right\}
\end{equation}
of $2\pi$-periodic trigonometric polynomials
\[\Pi_{2n(n+p)}^{\mathrm{trig}} := \left\{ q(t) = \sum_{m=0}^{2n(n+p)} \hspace{-2mm} a_m \cos(m t) + \hspace{-2mm} \sum_{m=1}^{2n(n+p) - 1} \hspace{-4mm} b_m \sin(mt):\quad a_m, b_m \in \Rr \right\}.\]

\begin{theorem} \label{Thm-isomorphism} The operator 
\[ \Emb: \PiL \to \Tripi, \quad \Emb \! P(t) = P(\Liscurve(t)), \quad t \in [0,2\pi],\]
defines an isometric isomorphism from the space $\left(\PiL, \langle \cdot, \cdot \rangle\right)$ onto
the space $\Tripi$ equipped with the inner product $ \ds \langle q_1, q_2 \rangle = \frac1{2\pi} \int_0^{2\pi} \!\!\! q_1(t) \overline{q_2(t)} \Dx{t}$. 
\end{theorem}

\begin{proof}
The system $\left\{ \hat{T}_i(x) \hat{T}_j(y): (i,j) \in \Gamma_{n,p}^L \right\}$ forms an orthonormal basis of the space $\PiL$. The image 
\begin{equation*} 
\Lagbastrig(t) := \LagbastrigE(t), \quad (i,j) \in \Gamma_{n,p}^L,
\end{equation*}
of this basis under the linear operator $\Emb$ is given by
\begin{equation} \label{eq-orthogonaltrigonometricbasis} 
                                          \Lagbastrig(t) =  \left\{\begin{array}{ll} 
                                          \, 1, & \text{if}\; (i,j) = (0,0), \\[1mm]
                                          \sqrt{2} \cos \left(i n t - i \frac{\pi}2 \right), & \text{if}\; j=0, i < 2(n+p), \\[1mm]
                                          \sqrt{2} \cos \left(j (n+p) t - j \frac{\pi}2 \right), & \text{if}\; i=0, j \leq 2n, \\[1mm]
                                          \,2 \cos\! \left(i n t - i \frac{\pi}2\right) \cos \left(j (n+p) t - j \frac{\pi}2 \right) ,\quad & \text{otherwise}. 
                                          \end{array} \right.
\end{equation}
For $(i,j) \in \Gamma_{n,p}^L$, $j \neq 2n$, the functions $\Lagbastrig(t)$ are trigonometric polynomials of degree less than $2n(n+p)$. 
The only trigonometric polynomial of exact degree $2n(n+p)$ is precisely $e_{0,2n}$. By the definition of the operator $\Emb$, the values 
$\Emb\! P (t_k)$ and $\Emb\! P (t_{k'})$, $t_k \neq t_{k'}$ coincide if $\Liscurve(t_k) = \Liscurve(t_{k'})$ is a self-intersection point of $\Liscurve$. 
This is precisely encoded in the constraints given in \eqref{def:spacetrigonometricpolynomials}. We can conclude
that $\Emb$ maps $\PiL$ into the space $\Tripi$. 

For polyonomials $P_1,P_2 \in \PiL$, the product polynomial $P_1 P_2$ is an element of the space $\Pi_{8n+4p-1}$ and satisfies $\langle P_1 P_2, T_{2(n+p)}(x)T_{2n}(y) \rangle = 0$. Therefore, 
by Lemma \ref{lem:int2Dto1D}, the set $\left\{ \Lagbastrig: \; (i,j) \in \Gamma_{n,p}^L \right\}$ is an orthonormal system in $\Tripi$, and thus, $\Emb$ is an isometric embedding 
from $\PiL$ into $\Tripi$.

Now, if we can show that the dimensions of $\PiL$ and $\Tripi$ coincide, the proof is 
finished. To this end, we consider in $\Pi_{2n(n+p)}^{\mathrm{trig}}$ the Dirichlet kernel
\[D_{2n(n+p)}(t) := \frac{1 + \cos (2n(n+p)t) + \ds 2 \!\!\!\!\!\! \sum_{k=1}^{2n(n+p)-1} \!\!\!\!\!\! \cos(kt)}{4n(n+p)} = \frac{\sin(2n(n+p) t ) \cos \frac{t}2 }{4n(n+p) \sin \frac{t}2} .\]
It is well known that the trigonometric polynomials 
\[ D_{2n(n+p)}^k(t) := D_{2n(n+p)}\left(t - t_k \right), \quad k = 1, \ldots, 4n(n+p),\] are precisely the Lagrange polynomials in the 
space $\Pi_{2n(n+p)}^{\mathrm{trig}}$ with respect to the points $t_k$, $k = 1, \ldots, 4n(n+p)$, i.e.
\[D_{2n(n+p)}^k \left( t_{k'} \right) = \delta_{k,k'}, \quad 1 \leq k,k' \leq 4n(n+p).\]
In general, the polynomials $D_{2n(n+p)}^k$ do not lie in the subspace $\Tripi$. However, we can define a basis for $\Tripi$ by using the linear combinations
\begin{equation} \label{eq:Lagrangebasispolynomialstrigonometric}
 l_\Aa(t) := \sum_{k = 1, \ldots, 4n(n+p): \atop k \in [\Aa]} D_{2n(n+p)}^k(t), \quad \Aa \in \Lub.
\end{equation}
Clearly, the polynomials $l_\Aa$ are elements of $\Tripi$, and $l_{\Aa}(t_k)$ is equal to one if $k \in [\Aa]$ and zero if $k \notin [\Aa]$. Also, by
the orthogonality of the functions $D_{2n(n+p)}^k$, we have $\langle l_\Aa, l_\Bb \rangle = 0$ if $\Aa \neq \Bb$. 
Therefore, the system $\{l_\Aa: \; \Aa \in \Lub \}$ forms an orthogonal basis of $\Tripi$ and $\dim (\Tripi) = |\Lub|$. This corresponds exactly with the dimension of the space $\PiL$. \qed
\end{proof}

\begin{theorem} \label{thm:interpolation problem}
For $\Aa = (x_\Aa, y_\Aa) \in \Lub$, the polynomials  $L_{\Aa} := \Emb^{-1} l_{\Aa}$ have the representation
\begin{equation} \label{eq:fundamentalpolynomialsofLagrangeinterpolation}
 L_{\Aa}(x,y) = w_\Aa \left( K_{n,p}^L(x,y; x_\Aa,y_\Aa) - \frac12 \hat{T}_{2n}(y) \hat{T}_{2n}(y_\Aa) \right)
\end{equation}
and are the fundamental polynomials of Lagrange interpolation in the space $\PiL$ on the point set $\Lub$, i.e. 
\[ L_{\Aa}(\Bb) = \delta_{\Aa,\Bb}, \quad \Aa, \Bb \in \Lub.\]
The interpolation problem \eqref{eq:interpolationproblem} has a unique solution in $\PiL$ and the 
interpolating polynomial $\Lagpol$ is given by
\begin{equation*} \label{eq:interpolationpolynomial}
\Lagpol (x,y) = \sum_{\Aa \in \Lub} f_\Aa L_\Aa(x,y). 
\end{equation*}
\end{theorem}

\begin{proof}
From the definition \eqref{eq:Lagrangebasispolynomialstrigonometric} of the trigonometric polynomials $l_{\Aa}$ and the mapping $\Emb$ it follows immediately that
the polynomials $L_{\Aa} = \Emb^{-1} l_{\Aa}$ satisfy $L_{\Aa}(\Bb) = \delta_{\Aa,\Bb}$ for $\Bb \in \Lub$. 
Moreover, since the trigonometric polynomials ${\{l_\Aa: \; \Aa \in \Lub \}}$ form an orthogonal basis of the space $\Tripi$, Theorem \ref{Thm-isomorphism} implies that the polynomials 
$\{L_\Aa: \; \Aa \in \Lub \}$ form an orthogonal basis of Lagrange polynomials for the space $\PiL$ as well. 

It remains to prove \eqref{eq:fundamentalpolynomialsofLagrangeinterpolation}. 
To this end, we compute the decomposition of the polynomials $l_\Aa$ in the basis $\Lagbastrig$ given in \eqref{eq-orthogonaltrigonometricbasis} and use 
the inverse of the operator $\Emb$ to obtain \eqref{eq:fundamentalpolynomialsofLagrangeinterpolation}. The proof will be given only for $\Aa\in \Lubblack$ having the
representation
\[ \Aa = (x_\Aa, y_\Aa) = \Liscurve \left( \frac{(2r+1)n + 2s (n+p) }{ 4 n (n+p)} 2\pi \right) 
                        = (-1)^{r+s} \left(z_{(2r+1)p}^{2(n+p)}, z_{2sp}^{2n} \right).\]
We first suppose that $\Aa\in \Lubblack$ is an interor point such
that the two points $k,k' \in [\Aa] \cap [1,4n(n+p)]$, $k \neq k'$ that represent the same $\Aa$ are given as 
\begin{align*}
k &= (2r+1) n + 2s (n+p) \; \mod \;4n(n+p),\\ 
k'&= (2r+1) n - 2s (n+p) \; \mod \;4n(n+p).
\end{align*}
Using simple trigonometric transformations, the basis function $l_\Aa$ can be written as
\begin{align*} l_\Aa(t) &= D_{2n(n+p)}^k(t) + D_{2n(n+p)}^{k'}(t) \\ &= \frac{2}{4n(n+p)} \left(  1 + \cos((2r+1)n \pi) \cos (2 n (n\!+\!p) t) \phantom{\sum_{m=1}^{2n(n+p)-1}} \right. \\
                        & \; + \left. 2 \hspace{-4mm} \sum_{m=1}^{2n(n+p)-1} \hspace{-4mm} \ts \cos \left( \frac{2sm \pi}{2n}\right) \! \left(  \cos \left( \frac{(2r+1)m\pi}{2(n+p)} \right)  \cos (m t) +
                       \sin \left( \frac{(2r+1)m\pi}{2(n+p)}\right)  \sin (m t) \right) \!\! \right)\!.                                                               
\end{align*}
Now, using the explicit expression \eqref{eq-orthogonaltrigonometricbasis} of the 
basis polynomials $\Lagbastrig$ and comparing the coefficients in the decomposition of $l_\Aa$,
we get the following formula for the inner product $\left\langle l_\Aa, \Lagbastrig \right\rangle = \frac{1}{2\pi} \int_0^{2\pi} l_\Aa (t) \Lagbastrig(t) \Dx{t}$, $(i,j) \in \Gamma_{n,p}^L$:
\begin{align*}
\left\langle  l_\Aa, \Lagbastrig \right\rangle &= \left\{
\begin{array}{ll}
\frac{2}{4n(n  +p)}, & \text{if}\; (i,j) = (0,0), \\
\frac{\sqrt{2}}{4n(n  +p)}, & \text{if}\;i = 0, j = 2n, \\
\frac{2 \sqrt{2}(-1)^{(r+s)j} }{4n(n  +p)} \cos \left( j \frac{2s p \pi }{2n}  \right), & \text{if}\; i = 0, j < 2n, \\
\frac{2\sqrt{2}(- 1)^{(r+s)i} }{4n(n  +p)} \cos \left( i \frac{(2r+1)p \pi }{2(n+p)}  \right), & \text{if}\; i \neq 0, j = 0,\\
\frac{4 (-1)^{(r+s)(i+j)}}{4n(n  +p)} \cos\!\left(  i \frac{(2r+1)p \pi}{2(n+p)} \! \right)\! \cos \! \left( j \frac{2s p \pi }{2n} \! \right), \quad &  \text{otherwise,}
\end{array} \right. \\[2mm]
 &=\frac{2}{4n(n\!+\!p)}\left\{ 
\begin{array}{ll}
\frac1{2}\, \hat{T}_{2n}(y_\Aa), & \text{if}\; i = 0, j = 2n, \\[2mm]
\hat{T}_{i}(x_\Aa) \hat{T}_j(y_\Aa),\quad & \text{if}\; (i,j) \in \Gamma_{n,p}^L \setminus (0,2n).
\end{array} \right.
\end{align*}
Therefore, $l_\Aa(t)$ can be decomposed as
\[l_\Aa(t) = \frac{\hat{T}_{2n}(y_\Aa)}{4n(n+p)} e_{0,2n}(t)  + \!\!\!\!\!\!
\sum_{\;\;(i,j) \in \Gamma_{n,p}^L \atop j \neq 2n} \!\!\!\!\!\! \frac{2 \hat{T}_{i}(x_\Aa) \hat{T}_j(y_\Aa)}{4n(n+p)}  \Lagbastrig(t). \]
Now, using the inverse mapping $\Emb^{-1}$ together with the definition of the reproducing kernel $K_{n,p}^L$, we can conclude:
\[L_\Aa(x,y) = \Emb^{-1} l_\Aa (x,y) = w_\Aa \left( K_{n,p}^L(x,y; x_\Aa, y_\Aa) - \frac12 \hat{T}_{2n}(y) \hat{T}_{2n}(y_\Aa) \right). \]
If $\Aa\in \Lubblack$ is a point on the boundary of the square $[-1,1]^2$, the number $k$ can be represented as $k = (2r+1)n$ and the basis function $l_\Aa$ is given as 
\begin{align*} l_\Aa(t) &= D_{2n(n+p)}^k(t) = \frac{1}{4n(n+p)} \left( 1 + \cos((2r+1)n \pi ) \cos (2 n (n+p) t) \phantom{\sum_{m=1}^{2n(n+p)-1}} \right. \\
                        & \left.\hspace{0.5cm} + \; 2 \!\! \!\!\!\! \sum_{m=1}^{2n(n+p)-1} \!\!\! \!\!\! \cos \left(\! \frac{2r+1}{2(n+p)} m \pi\! \right)  \cos (m t) 
                        + \sin \left( \! \frac{2r+1}{2(n+p)} m \pi \!\right) \sin (m t) \! \right).                                                                                  
\end{align*}
Now, similar calculations to the above yield \eqref{eq:fundamentalpolynomialsofLagrangeinterpolation} with the half sized weight function $w_\Aa = \frac1{4n(n+p)}$. 
Finally, for all points $\Aa \in \Lubwhite$, \eqref{eq:fundamentalpolynomialsofLagrangeinterpolation} can be obtained by analogous calculations 
using the representation \eqref{def:Luebeckpointsr} instead of \eqref{def:Luebeckpointsb}. \qed
\end{proof}

\begin{remark}
\eqref{eq:fundamentalpolynomialsofLagrangeinterpolation} has a remarkable resemblence to the Lagrange polynomials of the Padua points. 
For the Padua points, the analog statement of Theorem \ref{thm:interpolation problem} can be proved very elegantly by using ideal theory (cf. \cite{BosDeMarchiVianelloXu2007}). This approach was, however, not 
successful for the more general Lissajous nodes. Here, we had to use the isomorphism $\Emb$ and Theorem \ref{Thm-isomorphism} instead. 
\end{remark}

\section{A simple scheme for the computation of the interpolation polynomial}

In view of Theorem \ref{thm:interpolation problem}, the solution to the interpolation problem \eqref{eq:interpolationproblem} in $\PiL$ is given as
\[ \Lagpol(x,y) = \sum_{\Aa \in \Lub} w_\Aa f_\Aa \left( K_{n,p}^L(x,y; x_\Aa,y_\Aa) - \frac12 \hat{T}_{2n}(y) \hat{T}_{2n}(y_\Aa) \right).\]
The representation of the polynomial $\Lagpol(x,y)$ in the orthonormal Chebyshev basis $\{ \hat{T}_i(x) \hat{T}_j(y):\; (i,j) \in \Gamma_{n,p}^L\}$ can now be written as
\[ \Lagpol(x,y) = \sum_{(i,j) \in \Gamma_{n,p}^L} c_{i,j} \hat{T}_i(x) \hat{T}_j(y)\]
with the coefficients $c_{i,j} = \langle \Lagpol, \hat{T}_i(x) \hat{T}_j(y) \rangle$ given by
\[ c_{i,j} = \left\{ \begin{array}{ll}
                                                                          \ds \!\!\!\!\! \sum_{\;\;\;\Aa \in \Lub} \!\!\!\!\! w_\Aa f_\Aa \, \hat{T}_i(x_\Aa) \hat{T}_j(y_\Aa),\; & \text{if}\; (i,j) \in \Gamma_{n,p}^L \!\setminus\! (0,2n), \\[2mm]
                                                                          \ds \frac12 \!\!\!\!\! \sum_{\;\;\;\Aa \in \Lub} \!\!\!\!\! w_\Aa f_\Aa \, \hat{T}_{2n}(y_\Aa), & \text{if} \; (i,j) = (0,2n). \\
                                                                          \end{array}
\right.\]
Using a matrix formulation, this identity can be written more compactly. We introduce the coefficient matrix $\Ccc_{n,p} = (c_{ij}) \in \Rr^{2(n+p) \times 2n+1}$ by
\[ c_{ij} = \left\{ \begin{array}{ll} \langle \Lagpol, \hat{T}_i(x) \hat{T}_j(y),\quad & \text{if $(i,j) \in \Gamma_{n,p}^L$,}\\ 0, & \text{otherwise}.
                       \end{array} \right.\]
Next, we define the diagonal matrix 
\[ \Ddd_f(\Lub) = \diag \left( w_\Aa f_\Aa,\; \Aa \in \Lub \right) \in \Rr^{|\Lub|\times|\Lub|}.\]
Further, for a general finite set $\Sub \subset \Rr^2$ of points, we introduce the matrices 
\begin{align*} \Ttt_x(\Sub) &= \underbrace{\begin{pmatrix}
                   &\hat{T}_0(x_\Bb) & \\ \cdots & \vdots & \cdots \\ &\hat{T}_{2(n+p)-1}(x_\Bb)& 
                  \end{pmatrix}}_{\Bb \in \Sub} \in \Rr^{2(n+p) \times |\Sub|},\\ \Ttt_y(\Sub) &= \underbrace{\begin{pmatrix}
                   &\hat{T}_0(y_\Bb) & \\ \cdots & \vdots & \cdots \\ &\hat{T}_{2n}(y_\Bb)& 
                  \end{pmatrix}}_{\Bb \in \Sub}\in \Rr^{2n+1 \times |\Sub|}.
\end{align*}
Finally, we define the mask $\Mmm_{n,p} = (m_{ij}) \in \Rr^{2(n+p) \times 2n+1}$ by
\[ m_{i,j} = \left\{ \begin{array}{ll}
                     1, & \text{if}\; (i,j) \in \Gamma_{n,p}^L \setminus (0,2n), \\
                     1/2, \quad & \text{if}\; (i,j) = (0,2n),\\
                     0, & \text{if}\; (i,j) \notin \Gamma_{n,p}^L.
                     \end{array}
\right.\]
Now, the coefficient matrix $\Ccc_{n,p}$ of the interpolating polynomial can be computed as
\begin{equation} \label{eq:compcoeff} \Ccc_{n,p} = \left( \Ttt_x(\Lub) \Ddd_f(\Lub) \Ttt_y(\Lub)^T \right) \odot \Mmm_{n,p},\end{equation}
where $\odot$ denotes pointwise multiplication of the matrix entries. For an arbitrary point $\Bb \subset \Rr^2$, the evaluation $\Lagpol (\Bb)$ of the 
interpolation polynomial $\Lagpol$ at $\Bb$ is then given by
\begin{equation} \label{eq:comppol} \Lagpol(\Bb) = \Ttt_x(\Bb)^T \Ccc_{n,p} \Ttt_y(\Bb).\end{equation}

\begin{remark}
The matrix formulation in \eqref{eq:compcoeff} and \eqref{eq:comppol} is almost identical to the formulation
of the interpolating scheme of the Padua points given in \cite{CaliariDeMarchiVianello2008}. This is due to the similarity in the representation 
\eqref{eq:fundamentalpolynomialsofLagrangeinterpolation} of the Lagrange polynomials. The main 
difference between the schemes lies in the form of the mask $\Mmm_{n,p}$. The 
mask $\Mmm_{n,p}$ for the Lisa points has an asymmetric structure determined by the index set $\Gamma_{n,p}^L$, whereas the matrix is an upper left triangular matrix for the Padua points.
Two examples of such a structure are given in Figure \ref{fig:twobullets}. 
\end{remark}

\section{Numerical Simulations} \label{sec:numericalpart}
Based on the results derived in the last sections, we perform numerical simulations on the behaviour of the $\LubAbr$ points ($\mathrm{LS}$) in comparison to some 
already established point sets. Unless explicitly mentioned, we assume $p=1$ for all numerical simulations of the $\LubAbr$ points. For the comparison point sets, our 
focus is on the Xu points (XU) \cite{Xu1996} and the Padua points (PD) \cite{CaliariDeMarchiVianello2005}. Based on the Chebychev-Lobatto points given by \eqref{eq:chebLob}, 
and in correspondance to \eqref{eq:Luebeckpointsb2} and \eqref{eq:Luebeckpointsr2}, the odd Xu points $\mathrm{XU}_{2n+1}$ are defined as the union of the sets
\begin{align*}
	\textnormal{XU}^{\mathrm{b}}_{2n+1} &= \left\{(z^{2n+1}_{2i}, z^{2n+1}_{2j}):\; 0 \leq i \leq n,\; 0 \leq j \leq n\right\}, \\
	\textnormal{XU}^{\mathrm{w}}_{2n+1} &= \left\{(z^{2n+1}_{2i+1}, z^{2n+1}_{2j+1}):\; 0 \leq i \leq n,\; 0 \leq j \leq n\right\}, 		
\end{align*}
with the cardinality $|\textnormal{XU}_{2n+1}| = 2(n+1)^2$. In turn, the even Padua points $\mathrm{PD}_{2n}$ (2nd family) are defined as the union of the sets
\begin{align}
	\textnormal{PD}^{\mathrm{b}}_{2n} &= \left\{(z^{2n+1}_{2i+1}, z^{2n}_{2j}):\; 0 \leq i \leq n,\; 0 \leq j \leq n \right\}, \label{eq:Paduab}\\	
	\textnormal{PD}^{\mathrm{w}}_{2n} &= \left\{(z^{2n+1}_{2i}, z^{2n}_{2j+1}):\; 0 \leq i \leq n,\; 0 \leq j \leq n-1 \right\}. \label{eq:Paduaw} 	
\end{align}
The cardinality can be calculated as $|\textnormal{PD}_{2n}| = (n+1)(2n+1)$. The distributions of the $\LubAbr$, Xu and Padua points are shown for small degrees of $n$ in Figure~\ref{fig:points}. 
The point sets are introduced in such a way that an equally chosen $n$ results in a similar cardinality. 

\begin{figure}[htb]
	\centering
	\subfigure[	Point sets for $n=2$.]{\includegraphics[scale=0.72]{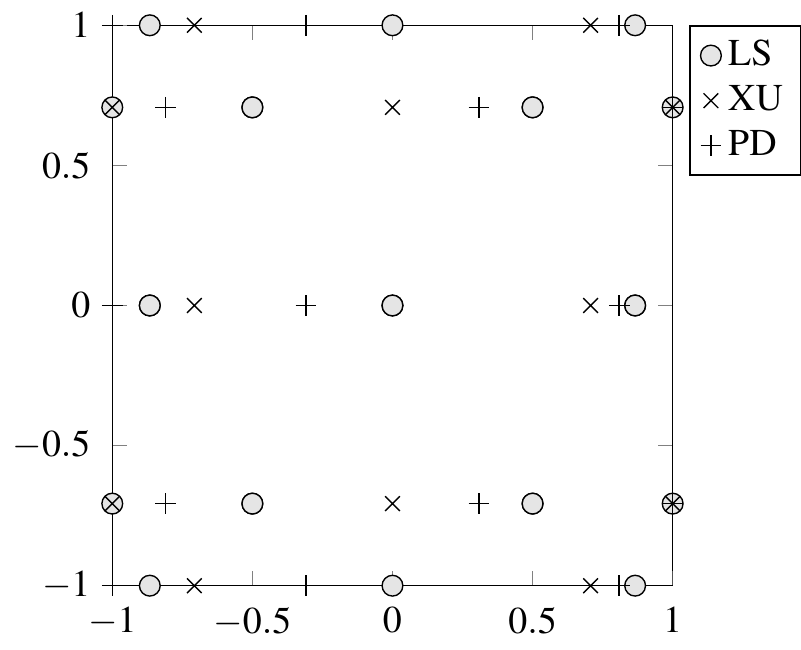}}
	\hfill	
	\subfigure[	Point sets for $n=5$.]{\includegraphics[scale=0.72]{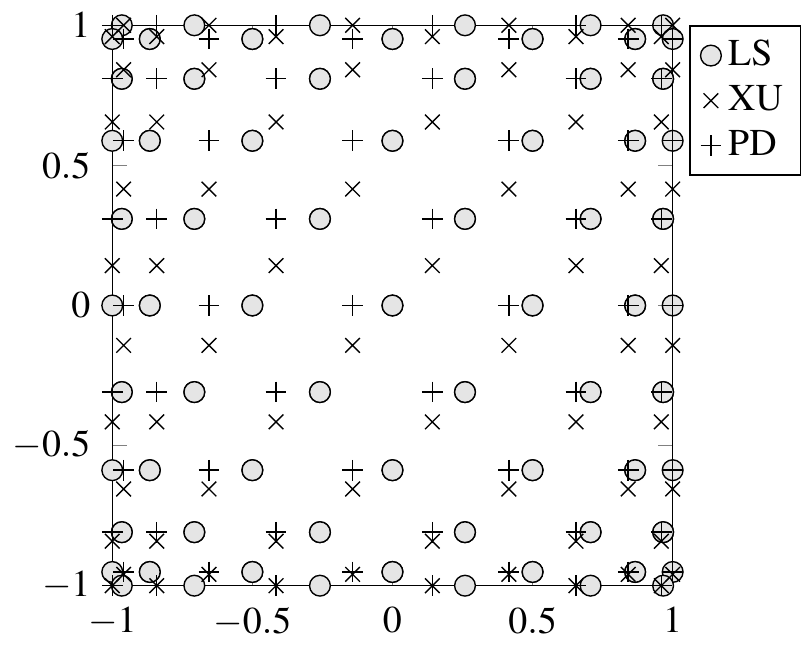}}
  	\caption{Visualizations of the $\LubAbr$ ($\mathrm{LS}$), Xu (XU) and Padua (PD) point sets.}
	\label{fig:points}
\end{figure}

The stability of the mapping $f \to \Lagpol$ is evaluated by means of the growth of the Lebesgue constant. Here, we calculate the values of the 
Lebesgue constant \[ \Lambda^{\textnormal{LS}}_{n,p} = \max_{\Bb \in [-1,1]^2} \sum_{\Aa \in \Lub} |L_\Aa(\Bb)|\] 
of the Lisa points up to a degree of $n=60$. We compare them with the least-squares fitting of the Lebesgue constant 
for the Padua and the Xu points. As shown 
in \cite{CaliariDeMarchiVianello2005}, it holds for the Padua points that ${\Lambda^{\textnormal{PD}}_{2n} = (\frac2{\pi} \log(2n + 1) + 1.1)^2}$ and as presented for the Xu points in 
\cite{BosCaliariDeMarchiVianello2006} that 
${\Lambda^{\textnormal{XU}}_{2n+1} = (\frac2{\pi} \log(2n + 2))^2}$.
Figure~\ref{fig:lebesgue}(a) indicates that the asymptotic growth of $\Lambda^{\mathrm{LS}}_{n,1}$ corresponds to the order $\mathcal O\left(\log^2 n \right)$ of the Lebesgue constant 
$\Lambda^{\mathrm{PD}}_{2n}$. In Figure~\ref{fig:lebesgue}(b) it is shown, how a variation of the parameter $p$ of the $\LubAbr$ points 
changes the growth of the Lebesgue constant. Here, we consider $p=\{1,3,5,7\}$ and excluded each entry for $n$ and $p$ not being relatively prime. In total, these numerical evaluations
suggest the conjecture that the Lebesgue constant of the Lisa points is of the same order $\mathcal O\left(\log^2 n \right)$ as the Lebesgue constant of the Padua and Xu points (see \cite{BosDeMarchiVianello2006}).

\begin{figure}[htb]
	\centering
	\subfigure[	$\LubAbr_{n,1}$, $\mathrm{XU}_{2n+1}$ and $\mathrm{PD}_{2n}$ point sets.]{\includegraphics[scale=0.65]{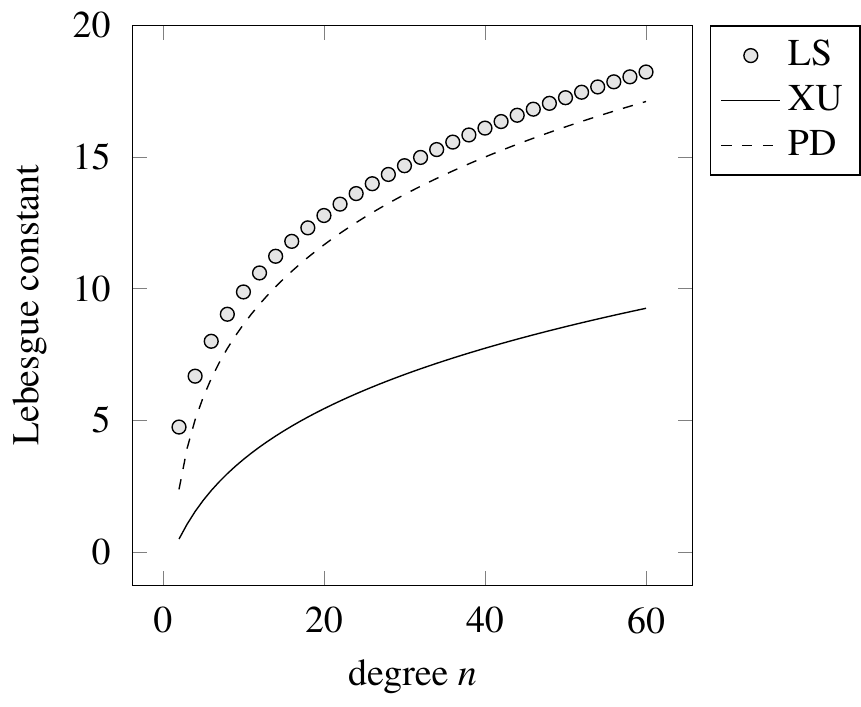}}
	\hfill	
	\subfigure[	$\LubAbr$ point sets for $p \in \{1,3,5,7\}$.]{\includegraphics[scale=0.65]{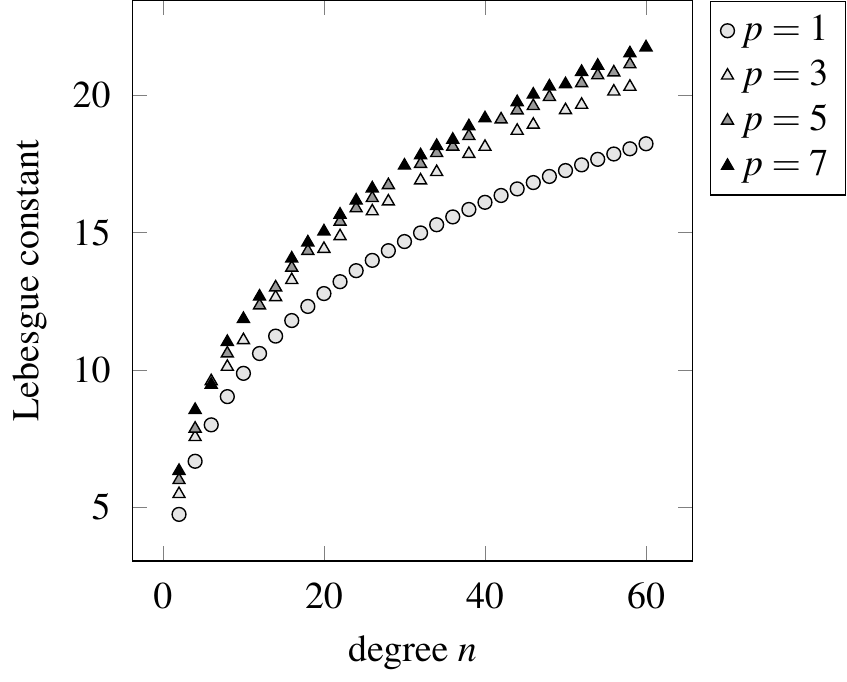}}
  	\caption{Lebesgue constants up to a degree of $n=60$ for the $\LubAbr$ points in comparison to the least-squares fitting of the Lebesgue constant of the Xu and Padua points.}
	\label{fig:lebesgue}
\end{figure}

For a further evaluation of the $\LubAbr$ points, we perform numerical interpolations with the Xu, Padua and $\LubAbr$ points on the 
Franke-Renka-Brown test set \cite{Franke1982,RenkaBrown1999}. In order to simulate the Xu points as well as the Padua points, the numerical algorithms presented in 
\cite{CaliariDeMarchiSommarivaVianello2011,CaliariVianelloDeMarchiMontagna2006} are used. The maximum interpolation errors are computed on a uniform grid
of $100\times100$ points defined in a region $\Omega = \left[0,1\right]\times\left[0,1\right]$. As mentioned above, the degree $n$ is defined to result in a
similar total number of points, i.e. a similar cardinality. For our simulations we take $n \in  \{5,10,20,30\}$. The results are shown in Table~\ref{tab:lub}--\ref{tab:pd}.
It can be seen that the maximum interpolation error of all three point sets shows a similar behaviour in terms of degree $n$, with respect to the chosen test function.
In terms of the $\LubAbr$ point sets, we evaluated the behaviour of $p$ in addition to the aforementioned comparisons. We can state that the influence of varying $p$, with
respect to the maximum interpolation error and the nodes used for the evaluation, is almost negligible.

\setlength{\tabcolsep}{1.2mm} 

\begin{table}[htb]
\caption{Interpolation errors for the points $\LubAbr_{n,1}$.}
\label{tab:lub}
\begin{center}
\begin{tabular}{ccccccccccccc}\hline \noalign{\smallskip} 
$n$ & \# & $F_1$ & $F_2$ & $F_3$ & $F_4$ & $F_5$ & $F_6$ & $F_7$ & $F_8$ & $F_9$ & $F_{10}$ \\  \noalign{\smallskip}\hline  \noalign{\smallskip} 
5 & 71 & 6\,E-2 & 4\,E-2 & 1\,E-3 & 6\,E-5 & 1\,E-2 & 3\,E-5 & 8\,E-1 & 2\,E-1 & 2\,E+1 & 4\,E-1 \\  \noalign{\smallskip}
10 & 241 & 7\,E-3 & 7\,E-3 & 1\,E-6 & 1\,E-10 & 2\,E-5 & 1\,E-8 & 1\,E-5 & 4\,E-3 & 4\,E-1 & 9\,E-2 \\  \noalign{\smallskip}
20 & 881 & 1\,E-6 & 2\,E-4 & 4\,E-12 & 5\,E-15 & 1\,E-13 & 1\,E-14 & 5\,E-14 & 1\,E-7 & 5\,E-6 & 4\,E-2 \\  \noalign{\smallskip}
 30 & 1921 & 3\,E-11 & 7\,E-6 & 3\,E-14 & 1\,E-14 & 4\,E-15 & 3\,E-14 & 2\,E-13 & 1\,E-13 & 9\,E-12 & 3\,E-2 \\  \noalign{\smallskip}
\hline 
\end{tabular}
\end{center} 
\end{table}

\begin{table}[htb]
\caption{Interpolation errors for the points $\mathrm{XU}_{2n+1}$.}
\label{tab:xu}
\begin{center}
\begin{tabular}{ccccccccccccc}\hline  \noalign{\smallskip}
$n$ & \# & $F_1$ & $F_2$ & $F_3$ & $F_4$ & $F_5$ & $F_6$ & $F_7$ & $F_8$ & $F_9$ & $F_{10}$ \\  \noalign{\smallskip}\hline  \noalign{\smallskip} 
5 & 72 & 8\,E-2 & 3\,E-2 & 1\,E-3 & 6\,E-5 & 1\,E-2 & 3\,E-4 & 6\,E-1 & 3\,E-1 & 3\,E+1 & 6\,E-1 \\  \noalign{\smallskip}
10 & 242 & 5\,E-3 & 6\,E-3 & 2\,E-6 & 1\,E-10 & 2\,E-5 & 1\,E-8 & 1\,E-5 & 5\,E-3 & 4\,E-1 & 1\,E-1 \\  \noalign{\smallskip}
20 & 882 & 1\,E-6 & 2\,E-4 & 5\,E-12 & 3\,E-15 & 1\,E-13 & 5\,E-15 & 3\,E-14 & 1\,E-7 & 5\,E-6 & 4\,E-2 \\  \noalign{\smallskip}
30 & 1922 & 3\,E-11 & 7\,E-6 & 1\,E-14 & 5\,E-15 & 3\,E-15 & 9\,E-15 & 4\,E-14 & 5\,E-14 & 9\,E-12 & 2\,E-2 \\  \noalign{\smallskip}

\hline 
\end{tabular} 
\end{center}
\end{table}

\begin{table}[htb]
\caption{Interpolation errors for the Padua points $\mathrm{PD}_{2n}$.}
\label{tab:pd}
\begin{center}
\begin{tabular}{ccccccccccccc} \hline  \noalign{\smallskip} 
$n$ & \# & $F_1$ & $F_2$ & $F_3$ & $F_4$ & $F_5$ & $F_6$ & $F_7$ & $F_8$ & $F_9$ & $F_{10}$ \\  \noalign{\smallskip}\hline  \noalign{\smallskip}
5 & 66 & 6\,E-1 & 4\,E-2 & 1\,E-3 & 6\,E-5 & 1\,E-2 & 3\,E-5 & 9\,E-1 & 2\,E-1 & 4\,E+1 & 5\,E-1 \\  \noalign{\smallskip}
10 & 231 & 6\,E-3 & 7\,E-3 & 3\,E-6 & 1\,E-10 & 2\,E-5 & 1\,E-8 & 2\,E-5 & 6\,E-3 & 7\,E-1 & 1\,E-1 \\  \noalign{\smallskip}
20 & 861 & 2\,E-6 & 2\,E-4 & 7\,E-12 & 3\,E-15 & 1\,E-13 & 4\,E-15 & 2\,E-14 & 1\,E-7 & 7\,E-6 & 4\,E-2 \\  \noalign{\smallskip}
30 & 1891 & 2\,E-11 & 7\,E-6 & 2\,E-14 & 6\,E-15 & 4\,E-15 & 2\,E-14 & 5\,E-14 & 6\,E-14 & 1\,E-11 & 2\,E-2 \\ \noalign{\smallskip}

\hline 
\end{tabular} 
\end{center}
\end{table}

\begin{acknowledgements}
The authors gratefully acknowledge the financial support of the German Federal Ministry of Education and Research 
(BMBF, grant number 13N11090), the German Research Foundation (DFG, grant number BU 1436/9-1 and ER 777/1-1), 
the European Union and the State Schleswig-Holstein (EFRE, grant number 122-10-004).
\end{acknowledgements}


\begin{thebibliography}{10}
\providecommand{\url}[1]{{#1}}
\providecommand{\urlprefix}{URL }
\expandafter\ifx\csname urlstyle\endcsname\relax
  \providecommand{\doi}[1]{DOI~\discretionary{}{}{}#1}\else
  \providecommand{\doi}{DOI~\discretionary{}{}{}\begingroup
  \urlstyle{rm}\Url}\fi

\bibitem{BogleHearstJonesStoilov1994}
Bogle, M.G.V., Hearst, J.E., Jones, V.F.R., Stoilov, L.: {Lissajous knots.}
\newblock J. Knot Theory Ramifications \textbf{3}(2), 121--140 (1994)

\bibitem{BosCaliariDeMarchiVianello2006}
Bos, L., Caliari, M., De~Marchi, S., Vianello, M.: {A Numerical Study of the Xu
  Polynomial Interpolation Formula in Two Variables.}
\newblock Computing \textbf{76}(3), 311--324 (2006)

\bibitem{BosDeMarchiVianelloXu2006}
Bos, L., Caliari, M., De~Marchi, S., Vianello, M., Xu, Y.: {Bivariate Lagrange
  interpolation at the Padua points: the generating curve approach.}
\newblock J. Approx. Theory \textbf{143}(1), 15--25 (2006)

\bibitem{BosDeMarchiVianello2006}
Bos, L., De~Marchi, S., Vianello, M.: {On the Lebesgue constant for the Xu
  interpolation formula.}
\newblock J. Approx. Theory \textbf{141}(2), 134--141 (2006)

\bibitem{BosDeMarchiVianelloXu2007}
Bos, L., De~Marchi, S., Vianello, M., Xu, Y.: {Bivariate Lagrange interpolation
  at the Padua points: The ideal theory approach.}
\newblock Numer. Math. \textbf{108}(1), 43--57 (2007)

\bibitem{CaliariDeMarchiSommarivaVianello2011}
Caliari, M., De~Marchi, S., Sommariva, A., Vianello, M.: {Padua2DM: fast
  interpolation and cubature at the Padua points in Matlab/Octave}.
\newblock Numer. Algorithms \textbf{56}(1), 45--60 (2011)

\bibitem{CaliariDeMarchiVianello2005}
Caliari, M., De~Marchi, S., Vianello, M.: {Bivariate polynomial interpolation
  on the square at new nodal sets.}
\newblock Appl. Math. Comput. \textbf{165}(2), 261--274 (2005)

\bibitem{CaliariDeMarchiVianello2008}
Caliari, M., De~Marchi, S., Vianello, M.: {Bivariate Lagrange interpolation at
  the Padua points: Computational aspects.}
\newblock J. Comput. Appl. Math. \textbf{221}(2), 284--292 (2008)

\bibitem{CaliariVianelloDeMarchiMontagna2006}
Caliari, M., Vianello, M., De~Marchi, S., Montagna, R.: {Hyper2d: a numerical
  code for hyperinterpolation on rectangles}.
\newblock Appl. Math. Comput. \textbf{183}(2), 1138--1147 (2006)

\bibitem{Cools1997}
Cools, R.: Constructing cubature formulae: the science behind the art.
\newblock Acta Numerica \textbf{6}, 1--54 (1997)

\bibitem{CoolsPoppe2011}
Cools, R., Poppe, K.: {Chebyshev lattices, a unifying framework for cubature
  with Chebyshev weight function.}
\newblock BIT \textbf{51}(2), 275--288 (2011)

\bibitem{Franke1982}
Franke, R.: Scattered data interpolation: Tests of some methods.
\newblock Math. Comp. \textbf{38}(157), 181--200 (1982)

\bibitem{GascaSauer2000b}
Gasca, M., Sauer, T.: {On the history of multivariate polynomial
  interpolation.}
\newblock J. Comput. Appl. Math. \textbf{122}(1-2), 23--35 (2000)

\bibitem{GascaSauer2000}
Gasca, M., Sauer, T.: {Polynomial interpolation in several variables.}
\newblock Adv. Comput. Math. \textbf{12}(4), 377--410 (2000)

\bibitem{Gleich2005Nature}
Gleich, B., Weizenecker, J.: Tomographic imaging using the nonlinear response
  of magnetic particles.
\newblock Nature \textbf{435}(7046), 1214--1217 (2005)

\bibitem{Gruettner2013BMT}
Gr{\"u}ttner, M., Knopp, T., Franke, J., Heidenreich, M., Rahmer, J., Halkola,
  A., Kaethner, C., Borgert, J., Buzug, T.M.: On the formulation of the image
  reconstruction problem in magnetic particle imaging.
\newblock Biomed. Tech. / Biomed. Eng. \textbf{58}(6), 583--591 (2013)

\bibitem{Harris2013}
Harris, L.A.: {Bivariate Lagrange interpolation at the Geronimus nodes.}
\newblock Contemp. Math. \textbf{591}, 135--147 (2013)

\bibitem{Kaethner2014IEEE}
Kaethner, C., Ahlborg, M., Bringout, G., Weber, M., Buzug, T.M.: Axially
  elongated field-free point data acquisition in magnetic particle imaging.
\newblock IEEE Trans. Med. Imag.  (2014).
\newblock Accepted for publication

\bibitem{Knopp2009PhysMedBio}
Knopp, T., Biederer, S., Sattel, T.F., Weizenecker, J., Gleich, B., Borgert,
  J., Buzug, T.M.: Trajectory analysis for magnetic particle imaging.
\newblock Phys. Med. Biol. \textbf{54}(2), 385--3971 (2009)

\bibitem{Lamm1997}
Lamm, C.: {There are infinitely many Lissajous knots.}
\newblock Manuscr. Math. \textbf{93}(1), 29--37 (1997)

\bibitem{Moeller1976}
M{\"o}ller, H.M.: {Kubaturformeln mit minimaler Knotenzahl.}
\newblock Numer. Math. \textbf{25}, 185--200 (1976)

\bibitem{MorrowPatterson1978}
Morrow, C.R., Patterson, T.N.L.: {Construction of algebraic cubature rules
  using polynomial ideal theory.}
\newblock SIAM J. Numer. Anal. \textbf{15}, 953--976 (1978)

\bibitem{RenkaBrown1999}
Renka, R.J., Brown, R.: Algorithm 792: Accuracy tests of acm algorithms for
  interpolation of scattered data in the plane.
\newblock ACM Trans. Math. Softw. \textbf{25}(1), 78--94 (1999)

\bibitem{Xu1996}
Xu, Y.: {Lagrange interpolation on Chebyshev points of two variables.}
\newblock J. Approx. Theory \textbf{87}(2), 220--238 (1996)

\end{thebibliography}
\end{document}